\documentclass{aic}

\usepackage{tikz}
\usetikzlibrary{decorations.pathreplacing} %positioning,arrows,shapes,decorations.markings,
\usepackage{cleveref}
\usepackage{xspace}

\newtheorem{thm}{Theorem}[section]
\newtheorem{lem}[thm]{Lemma}
\newtheorem{prop}[thm]{Proposition}
\newtheorem{claim}[thm]{Claim}
\newtheorem{conj}[thm]{Conjecture}
\newtheorem{cor}[thm]{Corollary}
\newtheorem{obs}[thm]{Observation}

\Crefname{thmsign}{Theorem}{Theorems}
\Crefname{lem}{Lemma}{Lemmas}
\Crefname{prop}{Proposition}{Propositions}
\Crefname{conj}{Conjecture}{Conjectures}
\Crefname{cor}{Corollary}{Corollaries}
\Crefname{prob}{Problem}{Problems}

\newcommand{\bN}{\mathbb{N}}
\newcommand{\cF}{\mathcal{F}}
\newcommand{\cG}{\mathcal{G}}
\newcommand{\tG}{{G}^*}
\newcommand{\ttG}{\widetilde{G}}

\newcommand{\ptg}[1]{CPR\xspace}%{complete pentagon-$#1$-partite\xspace}
\newcommand{\ptt}{pentagonal Tur\'an\xspace}
\newcommand{\floor}[1]{\left\lfloor #1 \right \rfloor}
\newcommand{\ceil}[1]{\left\lceil #1 \right \rceil}

\newcommand{\subs}{\subseteq}

\aicAUTHORdetails{%
  title = {Exact Stability for Tur\'an's Theorem}, %% please capitalize all significant words
  author = {D\'{a}niel Kor\'{a}ndi, Alexander Roberts, and Alex Scott},
  plaintextauthor = {Daniel Korandi, Alexander Roberts, Alex Scott},
  plaintexttitle = {Exact Stability for Turan's Theorem}, %%  title without math or LaTeX
  keywords = {stability theorem, extremal graph theory},
}   %%% END \aicAUTHORdetails

\aicEDITORdetails{%
   year={2021},
   number={9},
   received={5 December 2020},   % received date: example: 7 January 2017
   published={29 December 2021},  % published date
   doi={10.19086/aic.31079},      % XXX = number of paper, e.g. aic006 for paper#6
}   %%% END \aicEDITORdetails

\begin{document}

\begin{frontmatter}[classification=text]
%% EDITOR: this will force the keywords to appear right after the Abstract.
%%   If the abstract is too long and would force the keywords off the
%%   front page, please comment out % [classification=text] above
%%   This way the keywords will be floated on the bottom of the first page
%%   even though the Abstract spills over to the next page.

%%% AUTHOR: Title goes here.  This line is optional.  You must use it
%%   if title has footnote attached or requires nontrivial typesetting,
%%   e.g., inclusion of linebreaks to force nice layout.
% \title{Short Proof of R\"odl's $n^{\log\log n}$ Bound\titlefootnote{This is a footnote to the title}} %% please capitalize all significant words

%%% AUTHOR:
%%% List all authors. If you wish, place grant acknowledgements in \thanks.
%%% In brackets include a short tag for each author.
\author[dk]{D\'{a}niel Kor\'{a}ndi\thanks{Supported by SNSF Postdoc.Mobility Fellowship P400P2\_186686.}}
\author[ar]{Alexander Roberts}
\author[as]{Alex Scott}

%%% AUTHOR: Abstract goes here
\begin{abstract}
Tur\'an's Theorem says that an extremal $K_{r+1}$-free graph is $r$-partite. The Stability Theorem of Erd\H{o}s and Simonovits shows that if a $K_{r+1}$-free graph with $n$ vertices has close to the maximal $t_r(n)$ edges, then it is close to being $r$-partite.
In this paper we determine exactly the $K_{r+1}$-free graphs with at least $m$ edges that are farthest from being $r$-partite, for any $m\ge t_r(n) - \delta_r n^2$.
This extends work by Erd\H{o}s, Gy\H{o}ri and Simonovits, and proves a conjecture of Balogh, Clemen, Lavrov, Lidick\'y and Pfender.
\end{abstract}
\end{frontmatter}

\section{Introduction}

Tur\'an's classical theorem \cite{T41} from 1941 says that a $K_{r+1}$-free $n$-vertex graph maximizing the number of edges (an \emph{extremal graph}) is $r$-partite; the $r=2$ case was established earlier by Mantel \cite{M07}, in 1907.  The only extremal $n$-vertex graph is the \emph{Tur\'an graph} $T_r(n)$, the complete $r$-partite graph with parts of size $\floor{n/r}$ or $\ceil{n/r}$, which has $t_r(n) = \left(1-\frac{1}{r} + o(1)\right)\binom{n}{2}$ edges. Tur\'an's Theorem lay the foundations of extremal graph theory, and has been highly influential in the field ever since.

One of the early discoveries related to Tur\'an's Theorem was that if a $K_{r+1}$-free graph is ``close'' to extremal in the number of edges, then it must be ``close'' to the Tur\'an graph in its structure. Indeed, the famous Stability Theorem of Erd\H{o}s and Simonovits \cite{E66, S66} from the 1960s implies the following: if $G$ is a $K_{r+1}$-free $n$-vertex graph with $t_r(n)-o(n^2)$ edges, then it can be made into the Tur\'an graph $T_r(n)$ by changing only $o(n^2)$ edges.
It is of little surprise that this powerful structural description of near-extremal graphs has seen many important applications and consequences over the past decades (e.g. \cite{ABKS04,BBS04,S14,S09}). 

An alternative form of stability for Tur\'an's Theorem is to look at the distance from being $r$-partite (rather than the distance to a specific $r$-partite graph, namely the Tur\'an graph). Thus we are looking for a large $r$-partite subgraph, which is what is wanted for most applications.
The two problems are equivalent if we are only looking for a $o(n^2)$ bound on the distance. However, for graphs that are closer to extremal, we can obtain more structural information by measuring the distance from being $r$-partite. 
For example, if we move a constant number of vertices from a smallest vertex class to a largest vertex class of $T_r(n)$  then the resulting graph has $t_r(n)-O(1)$ edges but distance $\Omega(n)$ from the Tur\'an graph.   In contrast, a $K_{r+1}$-free graph on $n$ vertices with at least $t_r(n)-c_rn$ edges must already be $r$-partite.  This phenomenon was first studied by Simonovits \cite{S69} and later by many other authors \cite{B81,HT91,KP05,AFGS13,TU15}.   
A tight result was proved by Brouwer \cite{B81}:

\begin{thm}[\cite{B81}] \label{thm:rpartstab}
Let $r\ge 2$ and $n\ge 2r+1$ be integers. Every $K_{r+1}$-free graph with at least $t_r(n) - \floor{n/r} +2$ edges is $r$-partite.
\end{thm}

Let $f_r(n,t)$ be the smallest number such that any $K_{r+1}$-free graph $G$ with at least $t_r(n)-t$ edges can be made $r$-partite by deleting at most $f_r(n,t)$ edges.   For fixed $r$, \Cref{thm:rpartstab} tells us that $f_r(n,t)=0$ for $t\le n/r+O(1)$, while the Stability Theorem tells us that $f_r(n,t)=o(n^2)$ if $t=o(n^2)$.  But what happens in between?  Better estimates of this function have only been obtained fairly recently. In a short and elegant paper, F\"uredi \cite{F15} proved that $f_r(n,t)\le t$. Later, Roberts and Scott \cite{RS18} showed that $f_r(n,t) = O(t^{3/2}/n)$ when $t\le\delta n^2$, and that this bound is tight up to a constant factor (in fact, they proved much more general results for $H$-free graphs, where $H$ is edge-critical). Very recently, Balogh, Clemen, Lavrov, Lidick\'y and Pfender \cite{BCLLP} determined $f_r(n,t)$ asymptotically, and made a conjecture on its exact value.  The main aim of this paper is to prove their conjecture. 

When $r=2$, the exact stability problem was already solved by Erd\H{o}s, Gy\H{o}ri and Simonovits \cite{erdoscan'tmaths}: they proved that for $t\le n^2/20$ the worst triangle-free graph, defining $f_2(n,t)$, is a blowup of $C_5$. One can generalize this construction to obtain a family of $K_{r+1}$-free graphs with many edges as follows. Consider a complete $(r-1)$-partite graph with parts $Z,Z_3,\dots,Z_r$, and insert a blowup of $C_5$ on $Z$ with independent sets $X, Y_1, Y_2, Z_1, Z_2$ as in \Cref{fig:ptg} (so $Z= X\cup Y_1\cup Y_2\cup Z_1\cup Z_2$). We will call this a \emph{\ptt graph} if it further satisfies $|X| \le |Y_1| = |Y_2| \le |Z_i|$ for every $i\in[r]$, and each of the sets $X\cup Y_1\cup Z_1,X\cup Y_2\cup Z_2, Z_3,\dots,Z_r$ has size $\floor{\frac{n+|X|}{r}}$ or $\ceil{\frac{n+|X|}{r}}$.

\begin{figure}[t!]
\begin{center}
\begin{tikzpicture}[scale=1.25]

\draw[fill=white, rounded corners=2mm] (0,3.5) -- (0,4.5) -- (2,6.25) -- (9,6.25) -- (9,3.5) -- cycle;

%\draw[fill=black, fill opacity=.1] (0.3,1) -- (0.3,3.5) -- (1.2,3.5) -- (1.75,2.75) -- (2.3,3.5) --
%	(3.75,3.5) -- (4.25,2.75) -- (4.75,3.5) -- (8.5,3.5) -- (8.5,1) -- (0.25,1);
\draw[fill=black, fill opacity=.1] (0.3,1) rectangle (8.5,3.5);
\draw[fill=black, fill opacity=.1] (0.35,4.5) -- (2.35,5.25) -- (3.6,5.25) -- (1.25,4.5) -- cycle;
\draw[fill=black, fill opacity=.1] (2.35,5.25) -- (3.6,5.25) -- (8.5,4.5) -- (5,4.5) -- cycle;
\draw[fill=black, fill opacity=.1] (8.5,4.5) -- (5,4.5) -- (5,5.25) -- (8.5,5.25) -- cycle;
\draw[fill=black, fill opacity=.1] (5,5.25) -- (8.5,5.25) -- (3.6,4.5) -- (2.35,4.5) -- cycle;

\draw[fill=black, fill opacity=.1] (1.5,3.8) rectangle (2,4.3);

\draw[fill=white, rounded corners=2mm] (0,0) rectangle (9,1); 
\node at (4.5,0.5) {$|Z_4|=x+y+z$};
\draw[fill=white, rounded corners=2mm] (0,1.75) rectangle (9,2.75); 
\node at (4.5,2.25) {$|Z_3|=x+y+z$};
\draw[fill=white, rounded corners=2mm] (0.1,3.6) rectangle (1.5,4.5); 
\node at (0.8,4) {$|X| =x_{\phantom{2}}$};
\draw[fill=white, rounded corners=2mm] (2,3.6) rectangle (4,4.5); 
\node at (3,4) {$|Y_2|=y$};
\draw[fill=white, rounded corners=2mm] (4.5,3.6) rectangle (8.9,4.5); 
\node at (6.75,4) {$|Z_2|=z$};
\draw[fill=white, rounded corners=2mm] (2,5.25) rectangle (4,6.15); 
\node at (3,5.75) {$|Y_1|=y$};
\draw[fill=white, rounded corners=2mm] (4.5,5.25) rectangle (8.9,6.15); 
\node at (6.75,5.75) {$|Z_1|=z$};

\draw[decoration={brace,mirror,amplitude=5pt,raise=-2pt},decorate] (9.7,3.375) -- node[right=2pt] {$Z$} (9.7,6.375);

\end{tikzpicture}
%\vspace{-.5cm}
\caption{A \ptt graph with $r=4$}
\label{fig:ptg}
\end{center}
\end{figure}
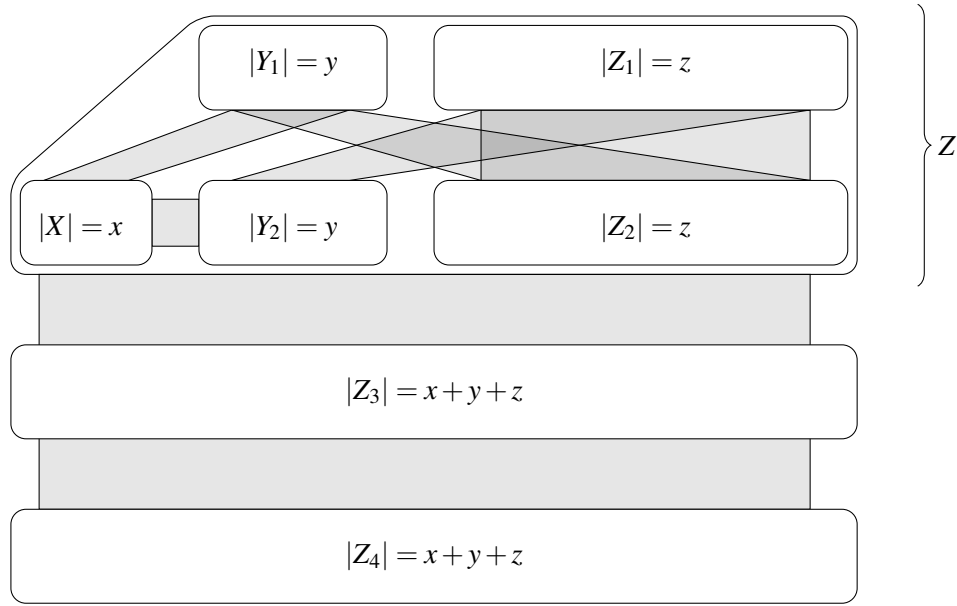

Balogh, Clemen, Lavrov, Lidick\'y and Pfender \cite{BCLLP} conjectured that $f_r(n,t)$ is witnessed by a \ptt graph if $t$ is small enough. 
Our main result is a proof of their conjecture.  
For a graph $G$ and integer $r\ge2$, let
$D_r(G)$ be the minimum number of edges that must be removed from $G$ to make it $r$-partite.  We prove the following theorem.

\begin{thm}\label{pentstab}
For every $r \ge 2$ there is a $\delta_r > 0$ such that the following holds: If $G$ is a $K_{r+1}$-free graph on $n$ vertices with $e(G) \ge t_r(n) - \delta_r n^2$ edges, then there is a \ptt graph $\tG$ on $n$ vertices with $e(\tG) \ge e(G)$ and $D_r(\tG) \ge D_r(G)$.
\end{thm}

\medskip

The rest of the paper is organized as follows. In \Cref{sec:tools}, we present a brief overview of the proof, and collect some necessary tools. We need a special argument when the number of edges in $G$ is very close to $t_r(n)$, and the short proof of this case is presented in \Cref{sec:dense}. \Cref{sec:mainproof} contains the general argument of the proof of \Cref{pentstab}. We finish the paper with some discussion and open problems in \Cref{sec:conclusion}.

We follow standard notation throughout. $G$ is always a simple graph with vertex set $V(G)$ and edge set $E(G)$. The number of edges is denoted by $e(G)=|E(G)|$. We write $\Gamma_G(v)\subs V(G)$ to denote the neighborhood of a vertex $v\in V(G)$, and $d_G(v)=|\Gamma_G(v)|$ to denote its degree. When the graph in question is clear, we may omit the subscript. For a set of vertices $S\subs V(G)$, we write $G-S$ for the subgraph induced on $V(G)\setminus S$. When $S=\{v\}$, we simply write $G-v$.

\section{Overview and tools} \label{sec:tools}

Given an $r$-partition of the vertices of a graph $G$, we say that an edge connecting different parts is \emph{crossing}, and an edge connecting vertices in the same part is \emph{internal}.
So $D_r(G)$ is the minimum number of internal edges in an $r$-partition of the vertices of $G$.

In their proof of the triangle-free case of \Cref{pentstab}, Erd\H{o}s, Gy\H{o}ri and Simonovits \cite{erdoscan'tmaths} start with a close to optimal bipartition of $G$, and construct a \ptt graph (in this case, a blowup of $C_5$) with the same number of internal edges, but more crossing edges. An important
 idea in their proof is to find a large matching of internal edges: as $G$ is triangle-free, this can be used to show that many crossing edges are missing from $G$.

Our proof for the general case follows a similar spirit, although we need to work harder to find the necessary missing edges when $K_{r+1}$ is forbidden instead of $K_3$.

We will need several estimates comparing Tur\'an numbers $t_r(n)$ for various $r$ and $n$. Recall that $t_r(n)$ is the number of edges in the Tur\'an graph $T_r(n)$, which is the complete $r$-partite graph on an \emph{$r$-equipartitioned} vertex set, i.e., when each part has size $\floor{n/r}$ or $\ceil{n/r}$. It is easy to see that $t_r(n) \ge t_r(n-1) + \frac{r-1}{r}(n-1)$, by adding a vertex to a smallest part of $T_r(n-1)$. Similarly, $t_r(n-1) \ge t_r(n) - \frac{r-1}{r}n$ can be obtained by deleting a vertex from a largest part of $T_r(n)$. 

The next lemma follows from these inequalities by iterating them, and by noting that $t_r(n)$ is the unique integer between $t_r(n-1) + \frac{r-1}{r}(n-1)$ and $t_r(n-1) + \frac{r-1}{r}n$.

\begin{lem} \label{lem:turan}
Let $r\ge 2$ and $n$ be integers. Then:
\begin{enumerate}
 \item $t_r(n) = t_r(n-1) + \ceil{\frac{r-1}{r}(n-1)} = t_r(n-1) + \floor{\frac{r-1}{r}n}$,
 \item $t_r(n') + \frac{r-1}{r}n(n-n') \ge t_r(n) \ge t_r(n') + \frac{r-1}{r}n'(n-n')$,\quad for every $n'\le n$,
 \item $\frac{r-1}{r}\binom{n+1}{2} \ge t_r(n) \ge \frac{r-1}{r}\binom{n}{2}$.
\end{enumerate}
\end{lem}

To find a large matching among the internal edges, we will use the following lemma, which follows easily from the Tutte-Berge formula (and is a special case of a theorem of Chv\'atal and Hanson \cite{CH76}). We include a sketch of the argument for completeness.

\begin{lem}\label{korandi}
Let $G$ be a graph on $n$ vertices with maximum degree $\Delta$ and let $k\ge 1$ be an integer. If $e(G) > (k-1) \Delta$ and $\Delta \ge 2k-1$, then $G$ contains a matching of size $k$.
\end{lem}
\begin{proof}[Proof sketch.]
If $G$ has no $k$-matching, then it contains a set $S$ such that $G-S$ has at least $n-2(k-1)+|S|$ odd components (note that perforce $|S| \le k-1$). The number of edges in this setup is maximized when $G-S$ is the union of $n-2(k-1)+|S|-1$ singletons and a $(2(k-1-|S|)+1)$-clique. Then $G-S$ induces $(k-1-|S|)(2(k-1-|S|)+1)\le (k-1-|S|)\Delta$ edges, and $S$ touches at most $|S|\Delta$ edges, so $G$ has at most $(k-1)\Delta$ edges, contradicting our assumption.
\end{proof}

For an integer vector $\mathbf{n} = (n_1,\dots,n_r) \in \bN^r$, let $K_{\mathbf{n}}$ be the complete $r$-partite graph with parts of size $n_1,\dots,n_r$.

The next lemma will be our main tool for bounding the number of missing crossing edges using the $K_{r+1}$-freeness of our graph. We will generally apply it to the neighborhood of a vertex.
This is a folklore result (see, for example, \cite{BMSW}), %[Lemma 3.3]
but we include a short proof for completeness.

\begin{lem}\label{folklore}
Let $r \ge 2$ and let $\mathbf{n} = (n_1,\dots,n_r) \in \bN^r$ be such that $n_1 \le n_2 \le \dots \le n_r$. Then any $K_r$-free subgraph of $K_{\mathbf{n}}$ contains at most $e\left(K_{\mathbf{n}}\right) - n_1n_2$ edges.
\end{lem}
\begin{proof}
There are exactly $\prod_{i=1}^{r} n_i$ copies of $K_r$ in $K_{\mathbf{n}}$. Each edge is contained in at most $\prod_{i=3}^{r} n_i$ of these copies, so a $K_r$-free subgraph must have at least $n_1n_2$ missing edges.
\end{proof}

We will also make use of the following classical result saying that every $K_{r+1}$-free graph with relatively large minimum degree is $r$-partite.

\begin{thm}[Andr\'{a}sfai-Erd\H{o}s-S\'{o}s \cite{aes}]\label{thm:aesos} %[Theorem 1.1]
Let $r \ge 2$ and let $G$ be a $K_{r+1}$-free graph $n$ vertices. If the minimum degree $\delta$ of $G$ is strictly greater than $\frac{3r-4}{3r-1}n$, then $G$ is $r$-partite.
\end{thm}

A \emph{blowup} $H=G[n_1,\dots, n_k]$ of a $k$-vertex graph $G$ is defined on vertex set $\bigcup_{i \in [k]} W_i$ with $|W_i|=n_i$, where the $W_i$ are disjoint, and $w\in W_i$ and $w'\in W_j$ are adjacent in $H$ if and only if $v_i$ and $v_j$ are adjacent in $G$. Note that every \ptt graph is a blowup $L_r[x,y,y,n_1,\dots,n_r]$, where $L_r$ is the graph whose first five vertices induce the pentagon $v_1v_2v_5v_4v_3$, and all other edges are present. Indeed, let us call such a blowup a \emph{complete pentagon-$r$-partite} (or \emph{\ptg{r}}) \emph{graph} if $x\le y\le n_i$ for every $i\in [r]$. A \ptt graph is then a \ptg{r} graph such that the numbers $x+y+n_1,x+y+n_2,n_3,\dots,n_r$ do not differ by more than 1 (i.e., each of them is equal to $\floor{\frac{n+x}{r}}$ or $\ceil{\frac{n+x}{r}}$).

The following statement tells us how to make blowups $r$-partite. We sketch the proof for completeness.
\begin{thm}[Erd\H{o}s-Gy\H{o}ri-Simonovits \cite{erdoscan'tmaths}]
 Let $H=G[n_1,\dots,n_k]$. Then one can delete $D_r(H)$ edges from $H$ to obtain $G'[n_1,\dots,n_k]$ for some $r$-partite subgraph $G'$ of $G$.
\end{thm}
\begin{proof}[Proof sketch.]
Take an $r$-partite subgraph of $H$ obtained by deleting $D_r(H)$ edges from $H$, and ``symmetrize'' it, i.e., for $i = 1,\dots,k$, carry out the following: Pick some $v \in W_i$ with $d(v)$ largest. Then for each $w\in W_i \setminus v$, change the edges touching $w$ so that its neighborhood $\Gamma(w)$ becomes the same as $\Gamma(v)$.

Through this process, the graph remains an $r$-partite subgraph of $H$, and the number of edges in it does not decrease (thus stays equal to $e(H)-D_r(H)$). At the end, we have  $\Gamma(v) = \Gamma(w)$ whenever $v$ and $w$ belong to the same blowup part $W_i$, so the resulting graph is the blowup of some $G'\subs G$ itself. 
\end{proof}
Deleting any edge of $L_r$ makes it $r$-partite, so we get the following.
\begin{cor} \label{lem:ptgcount}
If $G= L_r[x,y,y,n_1,\dots,n_r]$ is a \ptg{r} graph with $x\le y\le n_i$ for every $i\in[r]$, then $D_r(G)=xy$.
\end{cor}

This means that an optimal $r$-partition of a \ptg{r} graph (minimizing the number of internal edges) can be obtained by putting $Y_1\cup Z_1$ in the first part, $X\cup Y_2\cup Z_2$ in the second, and $Z_i$ in the $i$th part for every $i\ge 3$. Let us call this the \emph{standard $r$-partition} of such a graph.

As a benchmark, it will be helpful to understand roughly how many internal edges there are in the conjectured extremal graphs, so that we can cut short some edge cases in our analysis. 

\begin{lem} \label{lem:sampleptg}
For any integers $r\ge 2$, $n$ and $0\le s\le \frac{n}{r^4}$, there is a \ptg{r} graph $G$ with $n$ vertices and at least $t_r(n)-\frac{sn}{r}(1+1/r^3)$ edges such that $D_r(G) \ge \frac{\sqrt{s^3n}}{r^2}$.
\end{lem}
\begin{proof}
If $s=0$, then $G=T_r(n)$ satisfies the conditions, so we may assume that $s\ge 1$.
%, and hence $n\ge r^4$.

Let $t= \ceil{\frac{\sqrt{sn}}{r^2}}$. As $\sqrt{s} \le \frac{\sqrt{n}}{r^2}$, we have $s \le \frac{\sqrt{sn}}{r^2} \le t \le \frac{2\sqrt{sn}}{r^2} \le \frac{2n}{r^4}$. We claim that the graph $G=L_r[s,t,t,n_1,\dots,n_r]$ works if each of the numbers $n_1+t+\ceil{s/2}, n_2+t+\floor{s/2}, n_3, n_4,\dots, n_r$ is equal to $\ceil{n/r}$ or $\floor{n/r}$, in a non-increasing order. This graph is well-defined because, using $s\le t \le \frac{2n}{r^4}$ and $2 \le r$,
\[ t+\ceil{s/2} \le t+s \le 2t\le \frac{4n}{r^4} \le \frac{n}{2r}.\]
Moreover, since $\lfloor 2x \rfloor \ge 2 \lfloor x\rfloor$ for any $x>0$, this shows that $s\le t\le n_i$ for every $i\in [r]$, so by \Cref{lem:ptgcount}, $D_r(G) = st \ge \frac{\sqrt{s^3n}}{r^2}$.

To count the edges in $G$, let us split $X$ into two sets $X_1$ and $X_2$ of size $\ceil{s/2}$ and $\floor{s/2}$, respectively, and note that $(X_1\cup Y_1 \cup Z_1, X_2\cup Y_2\cup Z_2, Z_3, Z_4, \dots, Z_r)$ is an $r$-equipartition of the vertex set with exactly $st$ internal edges. There are $t_r(n)$ potential crossing edges, but $|X_1|(|X_2|+|Z_2|) + |X_2|(|X_1|+|Z_1|) - |X_1||X_2| + |Y_1||Y_2|$ of them are missing. 

%using $n/r^3\ge r^2 \ge 4$
Here $|Y_1||Y_2| = t^2 = \ceil{\frac{\sqrt{sn}}{r^2}}^2 \le \frac{sn}{r^4} + 2t - 1$ because $(\ceil{x}-1)^2 \le x^2$, and therefore $\ceil{x}^2 \le x^2 + 2\ceil{x} - 1$ for every $x\ge 1$. 
Also, $|X_1||X_2| = \floor {s/2}\ceil {s/2} = \floor {s^2/4}$. Finally, $|X_1|+|Z_1|$ and $|X_2|+|Z_2|$ are both at most $\ceil{n/r}-t \le \frac{n}{r}+1-t$, so we get $|X_1|(|X_2|+|Z_2|) + |X_2|(|X_1|+|Z_1|) \le s(\frac{n}{r}+1-t)$.

In total, this gives at least 
\[ st + t_r(n) - s \left( \frac{n}{r}+1-t \right) + \floor {\frac{s^2}{4}} - \left( \frac{sn}{r^4} + 2t - 1 \right) =
t_r(n) - \frac{sn}{r} - \frac{sn}{r^4} + 2st-2t + \floor {\frac{s^2}{4}} -s +1 \]
edges in $G$. We can see that this is at least $t_r(n)-\frac{sn}{r}(1+1/r^3)$ using the fact that $2st\ge 2t$ and $\floor{s^2/4} + 1 \ge s$ hold for every integer $s\ge 1$.
\end{proof}

\section{Very dense graphs} \label{sec:dense}

\Cref{thm:rpartstab} says that every $K_{r+1}$-free graph $G$ with very close to $t_r(n)$ edges is $r$-partite. The next lemma shows that $G$ is at most one vertex away from being $r$-partite, even if we allow slightly fewer edges.

\begin{lem} \label{lem:almostpartite}
Let $r\ge 2$, and suppose $G$ is a $K_{r+1}$-free graph on $n\ge 9r^4$ vertices with at least $t_r(n) - \frac{n}{r}(1+1/r^3)$ edges. Then there is a vertex $v\in V(G)$ such that $G-v$ is $r$-partite.
\end{lem}
\begin{proof}
If the minimum degree of $G$ is greater than $\frac{3r-4}{3r-1}n$, then by \Cref{thm:aesos}, $G$ itself is $r$-partite. Otherwise, there is a vertex $v$ of degree at most $\frac{3r-4}{3r-1}n$, and hence $G-v$ has
\begin{align*}
 e(G-v) &\ge t_r(n) - \frac{n}{r}(1+1/r^3) - \frac{3r-4}{3r-1}n \\
 &\ge t_r(n-1) + \frac{r-1}{r}n - \frac{r-1}{r} - \frac{n-1}{r} - \frac{1}{r} - \frac{n}{r^4} - \frac{3r-4}{3r-1}n \\
 &= t_r(n-1) - \frac{n-1}{r} + \frac{1}{r(3r-1)}n - \frac{n}{r^4} - 1 \\
 &\ge t_r(n-1) - \frac{n-1}{r} + \frac{n}{3r^4} -1 \\
 &\ge t_r(n-1) - \frac{n-1}{r}  + 2 
\end{align*}
edges, where we used $t_r(n) \ge t_r(n-1) + \frac{r-1}{r}(n-1)$ from \Cref{lem:turan} in the second line, $3r^2 \le 3r^4/4$ in the fourth, and $n\ge 9r^4$ in the fifth. But then $G-v$ is $r$-partite by \Cref{thm:rpartstab}.
\end{proof}

This structural lemma allows us to establish our main result when the number of edges is very close to extremal.

\begin{thm} \label{thm:verydense}
Let $r\ge 2$ and $n\ge 2^8r^4$, and suppose $G$ is a $K_{r+1}$-free graph with $n$ vertices and at least $t_r(n)-\frac{n}{r}(1+1/r^3)$ edges. Then there is a \ptg{r} graph $\tG$ such that $D_r(\tG) \ge D_r(G)$ and $e(\tG) \ge e(G)$. 
\end{thm}
\begin{proof}
If $G$ is $r$-partite, then we can just take $\tG=T_r(n)$, so let us assume that $G$ is not $r$-partite. By \Cref{lem:almostpartite}, there is a vertex $v$ such that $G-v$ is $r$-partite, say with parts $U_1, \dots, U_r$ of size $n_1,\dots, n_r$. Let $a_i$ be the number of neighbors of $v$ in $U_i$. We may assume that $a_1\le \dots\le a_r$. Then clearly, $1\le D_r(G)\le a_1$. We claim that $\tG = L_r[1,a_1,a_1,n_1-a_1,n_2-a_1,n_3,n_4\dots,n_r]$ works.

To show this, note that $G$ has 
\begin{equation} \label{eq:denseedges}
 e(G) \le \sum_{i<j} n_in_j - a_1a_2 + \sum_{i\in [r]} a_i 
\end{equation}
edges. This is because there are $\sum_{i<j} n_in_j$ potential edges in the $r$-partite graph induced by $U = U_1\cup\dots \cup U_r$, but the neighborhood of $v$ is $K_r$-free, so by \Cref{folklore}, at least $a_1a_2$ of these edges are missing. The number of edges in $G$ not induced by $U$ is precisely $\sum_{i\in [r]} a_i$.

On the other hand, 
\[
e(\tG) = \sum_{i<j} n_in_j - a_1^2 + 2a_1 + \sum_{i=3}^r n_i \ge  \sum_{i<j} n_in_j - a_1^2 + a_1-a_2 + \sum_{i\in [r]} a_i.
\]
As $a_1a_2 \ge a_1^2-a_1+a_2$ for any positive integers $a_2\ge a_1$, we get $e(\tG)\ge e(G)$.

To conclude the argument, it is enough to prove that $n_i\ge 2a_1$ for every $i\in[r]$. Indeed, this will establish that $\tG$ is a \ptg{r} graph, and, using \Cref{lem:ptgcount}, imply that $D_r(\tG) = a_1$. We can show this through a fairly straightforward calculation.

\medskip
As the number of edges in an $r$-partite graph is maximized by the Tur\'an graph, we have $\sum_{i<j} n_in_j \le t_r(n)$. Combining this with \eqref{eq:denseedges}, we get $e(G) \le t_r(n) - a_1a_2 + n$. But we assumed that $e(G) >  t_r(n) - n$, so $a_1\le \sqrt{2n}$.

On the other hand, suppose that $n_{i'} \le 3\sqrt{n}$ for some $i'\in [r]$, and let us define $\mathbf{n} = (n_1,\dots,n_r)$ and $\mathbf{n}'=(n_1,\dots,n_{i'-1},n_{i'+1},\dots,n_r)$. Once again, the maximality of Tur\'an graphs gives
\[
\sum_{i<j} n_in_j = e\left(K_{\mathbf{n}}\right) \le e\left(K_{\mathbf{n}'}\right) + n_{i'}n \le t_{r-1}(n-n_{i'}) + n_{i'}n \le t_{r-1}(n) + n_{i'}n.
\]
We can therefore further bound \eqref{eq:denseedges} as
\[
e(G) \le t_{r-1}(n) + 3n^{3/2} + n \le \frac{r-2}{r-1}\cdot \frac{n^2}{2} + 4n^{3/2} \le \frac{r-1}{r}\cdot \frac{n^2}{2} - \frac{n^2}{2r^2} + \frac{n^2}{4r^2} \le t_r(n) - n,
\]
using $n\ge 2^8r^4$ and $\frac{r-1}{r}\cdot \frac{n^2}{2}+n \ge t_r(n)\ge \frac{r-1}{r}\cdot \frac{n^2}{2} - n$ from \Cref{lem:turan}. But this contradicts our assumption on $e(G)$, so indeed, $n_i\ge 3\sqrt{n} \ge 2a_1$ for every $i\in[r]$.
\end{proof}

\section{Proof of \Cref{pentstab}}  \label{sec:mainproof}

It will be more convenient for us to prove the following, slightly weaker analog of \Cref{pentstab}.

\begin{thm}\label{thm:weakpentstab}
For every $r \ge 2$ there is a $\delta_r > 0$ such that the following holds: If $G$ is a $K_{r+1}$-free graph on $n$ vertices with $e(G) \ge t_r(n) - \delta_r n^2$ edges, then there is a \ptg{r} graph $\tG$ on $n$ vertices with $e(\tG) \ge e(G)$ and $D_r(\tG) \ge D_r(G)$.
\end{thm}
This statement easily implies the full theorem:
\begin{proof}[Proof of \Cref{pentstab}] \Cref{thm:weakpentstab} shows the existence of a \ptg{r} graph $\tG$, such that $e(\tG) \ge e(G)$ and $D_r(\tG) \ge D_r(G)$. Let us choose such a $\tG$ so that $e(\tG)$ is maximum. We claim that this $\tG$ is in fact a \ptt graph.

We know that $\tG=L_r[x,y,y,n_1,\dots,n_r]$ such that $x\le y\le n_i$ for every $i\in [r]$. Note that $e(\tG)\le t_r(n)-y^2-xn_1+xy\le t_r(n)-y^2$, so if $\delta_r<r^{-10}$, then $y\le \frac{n}{4r}$. To show that $\tG$ is a \ptt graph, we just need to check that the numbers $x+y+n_1,x+y+n_2,n_3,\dots,n_r$ do not differ by more than 1. Suppose that the $i$-th of these quantities is the largest among them, and the $j$-th is the smallest. If their difference was at least 2, then the graph $\ttG=L_r[x,y,y,n_1,\dots,n_i-1,\dots,n_j+1,\dots,n_r]$ would have more edges than $\tG$. Also, $x+y+n_i\ge \frac{n}{r}$ and $x\le y\le \frac{n}{4r}$, so $\ttG$ is a \ptg{r} graph with $D_r(\ttG)=xy = D_r(\tG)$. This contradicts the maximality of $\tG$ and establishes the theorem.
\end{proof}

Our proof of \Cref{thm:weakpentstab} divides into two main parts: defining a \ptg{r} graph $\tG$ based on our $G$, and comparing the number of edges in $G$ and $\tG$.
In the first part of the proof, we find an appropriate $r$-partition of $G$, with a large enough matching of internal edges, and use structural considerations to construct a $\tG$ that has at least as many \emph{internal} edges in its standard $r$-partition as $G$. Then in the second part, we use the $K_{r+1}$-freeness of $G$ to prove that it misses many of its \emph{crossing} edges, and ultimately show that $\tG$ has more crossing edges in its standard $r$-partition.

\subsection{The candidate \ptg{r} graph $\tG$}

\begin{proof}[Proof of \Cref{thm:weakpentstab}]

We will start with defining an $r$-partition on $G$.

Let $\delta_r = r^{-60}$, and suppose our $K_{r+1}$-free graph $G=(V,E)$ has $t_r(n)-\delta n^2$ edges for some $\delta \in (0,r^{-60})$. We may assume that $\delta n^2\ge 1$, and hence $n\ge \delta^{-1/2} \ge r^{20}$. Now if $\delta n^2 \le \frac{n}{r}(1+1/r^3)$, then we can apply \Cref{thm:verydense}, noting that $n\ge r^{20} \ge 2^8r^4$, to obtain the desired $\tG$. So we may also assume that $\delta n^2 > \frac{n}{r}(1+1/r^3)$, and in particular, $n\ge \frac{1}{\delta r}$.

\medskip
We first show that $G$ contains a large induced subgraph with high minimum degree.

\begin{prop}\label{weeding}
There is a vertex subset $S \subs V$ with $|S| \le 2\delta r^{10}n$ such that for all $v \in V\setminus S$,
\[ d_{G - S}(v) \ge n\left(\tfrac{r-1}{r} - r^{-10}\right). \]
\end{prop}
\begin{proof}
Let us iteratively remove vertices of degree less than $n(\frac{r-1}{r} - r^{-10})$. If this procedure stops with at most $2\delta r^{10}n$ removals, then we are done by choosing $S$ to be the set of removed vertices. So suppose otherwise, and let $B$ be the set of the first $\ceil{ 2\delta r^{10} n}$ vertices deleted. Then the number of edges in the graph $J = G - B$ can be bounded by
\[ 		e(J) \ge e(G) - n\left(\tfrac{r-1}{r} - r^{-10}\right)|B| = t_r(n) - \delta n^2 - n\left(\tfrac{r-1}{r} - r^{-10}\right)|B|.
\]
By \Cref{lem:turan}, we have $t_r(n) \ge t_r(n-|B|) + \frac{r-1}{r}(n-|B|)|B|$, and hence
\[
	e(J) \ge t_r(|J|) - \delta n^2 + r^{-10} n |B| - \tfrac{r-1}{r}|B|^2.
\]
Note that $|B|\ge 2$ (as $2\delta r^{10} n \ge 2r^9>1$), so $2\delta r^{10} n \le |B| \le 4\delta r^{10}n$. Using $1 > \delta r^{60}> 16\delta r^{20}$, this yields
\[ r^{-10} n|B| \ge 2\delta n^2 > \delta n^2 + 16 \delta^2 r^{20}n^2 \ge \delta n^2 + |B|^2.  \]
But then $e(J) > t_r(|J|)$, contradicting the fact that $J$ is $K_{r+1}$-free.
\end{proof}

\Cref{thm:aesos} implies that $G - S$ is $r$-partite. Let $U_1\cup \dots\cup U_r$ be an $r$-partition of $G-S$. By the minimum degree condition of $G - S$, every vertex $x\in U_i$ has at least $n(\frac{r-1}{r}-r^{-10})$ neighbors in $G - S - U_i$, so $|U_i|\le n(\frac{1}{r} + r^{-10}) - |S|$ for each $i$. On the other hand, $|U_i| \ge n  - |S| - \sum_{j\ne i} |U_j|$, so we get that for every $i$,
\begin{equation} \label{usizebound}
	|U_i| \ge n\left(\tfrac{1}{r} - (r-1)r^{-10}\right). 
\end{equation}
This also means that the neighborhood of each vertex in $U_i$ misses at most $r^{-9}n$ vertices in $\bigcup_{j\neq i}U_j$ and so the number of crossing edges missing between the $U_i$ is at most $r^{-9}n^2$.

\medskip
Now let us extend this partition into an $r$-partition $V = V_1\cup \dots\cup V_r$ of the entire vertex set of $G$ that maximizes the number of crossing edges, assuming $U_i \subs V_i$. In particular, each vertex of $S$ has at most as many neighbors in its own part as in any other part, i.e., for $s\in S\cap V_i$,
\begin{equation} \label{eq:maxcut}
	|\Gamma(s) \cap V_i| = \min_{j \in [r]}|\Gamma(s) \cap V_j|.
\end{equation}

Let us define $\Delta$ to be the maximum internal degree of $G$ in this partition, i.e.,
\[
	\Delta = \max_{i\in [r]} \max_{v\in V_i} |\Gamma(v) \cap V_i|
\]

\begin{claim} \label{Ddef}
We may assume that $\Delta$ is the internal degree of some vertex $u\in S$, and that
\[ 6|S|\le \Delta \le 2r^{-4.5} n. \]
\end{claim}
\begin{proof}
Note that all internal edges are incident with $S$ and so $D_r(G) \le |S| \Delta$. If $\Delta$ is smaller than $6|S|$, then $D_r(G)\le 6|S|^2\le 24\delta^2 r^{20}n^2$. We claim that there is a \ptg{r} graph $\tG$ with at least $t_r(n)-\delta n^2$ edges such that $D_r(\tG)$ is larger than this. Indeed, apply \Cref{lem:sampleptg} with $s=\floor{\frac{\delta rn}{1+1/r^3}}$ to obtain the graph $\tG$ with at least $t_r(n)-\delta n^2$ edges and $D_r(\tG)\ge \frac{\sqrt{s^3n}}{r^2}$. Our previous assumption that $\delta n^2 >\frac{n}{r}(1+1/r^3)$ implies that $s\ge 1$, and therefore $s\ge \frac{\delta rn}{4}$. This means that
\[ 
D_r(\tG) \ge \frac{\delta^{3/2}r^{3/2}n^2}{8r^2} > \frac{\delta^2 r^{29} n^2}{8} > 24\delta^2 r^{20}n^2 \ge D_r(G),
\]
as required (we used $1>\sqrt{\delta}r^{30}$ and $r\ge 2$).

So we may assume that $\Delta\ge 6|S|$. In particular, as the internal degree of each vertex in $V\setminus S$ is at most $|S|$, a vertex of maximum internal degree $\Delta$ must lie in $S$. Let $u$ be any such vertex.

Now we see from \eqref{eq:maxcut} that $|\Gamma(u) \cap U_i| \ge \Delta -|S| \ge \frac{5\Delta}{6}$ for every $i \in [r]$. Since $\Gamma(u)$ is $K_r$-free, \Cref{folklore} tells us that there are at least $\left(\tfrac{5\Delta}{6}\right)^2 \ge \Delta^2/2$ crossing edges missing between the $U_i$. On the other hand, we have seen that there are at most $r^{-9}n^2$ such edges missing, so $\Delta \le 2r^{-4.5}n$.
\end{proof}

Let $u \in S$ be the vertex from \Cref{Ddef}. By \eqref{eq:maxcut}, it has at least $\Delta$ neighbors in each $V_i$. For each $i \in [r]$, fix a set $P_i \subs \Gamma(u) \cap V_i$ with $|P_i| = \Delta$.

\medskip
We now come to finding a suitable matching consisting of internal edges. Let $H = \bigcup_{i \in [r]}G[V_i]$ be the subgraph of $G$ containing only the internal edges. Then $H$ has at most $\Delta|S|$ edges and maximum degree $\Delta$. Let $k=\ceil{\frac{e(H)}{\Delta}}$ and note that $k \le |S|$, so $\Delta \ge 6|S| \ge 2k$. Therefore, by \Cref{korandi}, we can find a matching $M$ of size $k$ in $H$.

For each $i \in [r]$, let $M_i = M[V_i]$ be the set of matching edges in $V_i$. Further split each $M_i$ into three sets $M_i = A_i \cup B_i \cup C_i$ according to the matching pairs' interaction with $P_i$:
\begin{align*}
	A_i &= \left\{uv \in M_i : u,v \notin P_i\right\}, \\
	B_i &= \left\{uv \in M_i : u \in P_i, v \notin P_i\right\}, \\
	C_i &= \left\{uv \in M_i : u,v \in P_i\right\}.
\end{align*}
Then define $a_i = |A_i|$, $b_i = |B_i|$ and $c_i = |C_i|$, and set $a = \sum_{i\in [r]}a_i$, $b = \sum_{i\in [r]}b_i$, and $c = \sum_{i\in [r]}c_i$ (so we have $k=a+b+c$). Note that if $V^A_i, V^B_i, V^C_i$ and $V^M_i$ denote the vertex sets of the matchings $A_i,B_i,C_i$ and $M_i$ respectively, then $|V^A_i| = 2a_i$, $|V^B_i| = 2b_i$ and $|V^C_i| = 2c_i$. We denote the unions over $i\in [r]$ by $V^A, V^B, V^C$ and $V^M$, so $|V^M| = 2k$ (see \Cref{fig:struct}).

\medskip
Finally, we set $R_i = V_i \setminus (P_i \cup V^M_i)$ and $\kappa_i = |R_i|$. With this notation at hand, we note that $|V_i| = \kappa_i + \Delta + 2a_i+b_i$ for each $i \in [r]$. To bound $\kappa_i$ from below, recall that $U_i\subs V_i$ is an independent set, so at most $k\le |S|$ of its vertices are covered by $M$. So by \eqref{usizebound}, $\delta < r^{-60}$, \Cref{weeding} and \Cref{Ddef}, we have
\begin{align*}
	\kappa_i \ge |U_i| - |S| - \Delta &\ge n\left(\frac{1}{r} - (r-1)r^{-10}\right) - 2\delta r^{10} n - 2r^{-4.5}n \nonumber \\
	&\ge n\left(r^{-1} - r^{-9} - 2r^{-50} - 2r^{-4.5}\right) \nonumber \\
	&\ge r^{-4.5}n\left(r^{3.5} - 3\right) \ge 8r^{-4.5}n \ge 4\Delta.
\end{align*}
We may assume without loss of generality that $\kappa_1 \le \kappa_2 \le \dots \le \kappa_r$. Together with \Cref{Ddef}, we get the following relationship between our quantities, which we will use throughout the rest of the proof:
\begin{equation} \label{eq:consts}
 \kappa_r \ge \dots \ge \kappa_2\ge \kappa_1 \ge 4\Delta \ge 24k.
\end{equation}

\begin{figure}[t!]
\begin{center}
\begin{tikzpicture}[scale=1.05]

\foreach \j in {1,...,4}
{
	\draw[rounded corners=2mm] (5,0 - 2*\j) rectangle (10,1.2 - 2*\j); 

	\draw[rounded corners=2mm] (1,0 - 2*\j) rectangle (4,1.2 - 2*\j); 
	
	\draw[gray,thick] (.7,.9 - 2*\j) -- (1.3,.9 - 2*\j) ;
	\draw[gray,thick] (.7,.6 - 2*\j) -- (1.3,.6 - 2*\j) ;
	\draw[gray,thick] (.7,.3 - 2*\j) -- (1.3,.3 - 2*\j) ;

	\draw[opacity=.5, rounded corners=2mm] (.5,0 - 2*\j) rectangle (4,1.2 - 2*\j); 

	\draw[blue,densely dotted,thick] (1.5,.75 - 2*\j) -- (2.1,.75 - 2*\j) ;
	\draw[blue,densely dotted,thick] (1.5,.45 - 2*\j) -- (2.1,.45 - 2*\j) ;

	\draw[red,densely dashed,thick] (-.6,.75 - 2*\j) -- (0,.75 - 2*\j) ;
	\draw[red,densely dashed,thick] (-.6,.45 - 2*\j) -- (0,.45 - 2*\j) ;

	\draw[densely dotted,rounded corners=2mm] (-1,-.1 - 2*\j) rectangle (10.2,1.3 - 2*\j); 

}

\draw[decoration={brace,amplitude=5pt,raise=-2pt},decorate] (5,-.35) -- node[above=2pt] {$R$} (10,-.35);
\draw[gray,decoration={brace,amplitude=5pt,raise=-2pt},decorate] (.5,-.35) -- node[above=2pt] {$Q = P \cup V^B$} (4,-.35);
\draw[red,decoration={brace,amplitude=5pt,raise=-2pt},decorate] (-.7,-.35) -- node[above=2pt] {$V^A$} (.1,-.35);
\draw[decoration={brace,amplitude=5pt,mirror,raise=-2pt},decorate] (-.7,-8.5) -- node[below=2pt] {$V^M$} (2.2,-8.5);
\draw[decoration={brace,amplitude=5pt,mirror,raise=-2pt},decorate] (11,-.1-2) -- node[right=2pt] {$V_1$} (11,1.3-2);

\node at (7.5,0.6 - 2) {$R_1$};
\node at (3,0.6 - 2) {$P_1$};

{\footnotesize
\node at (7.5,0.6 - 2*3) {$|R_i| = \kappa_i$};
\node at (3,0.6 - 2*3) {$|P_i| = \Delta$};
\node at (12.5,0.6 - 2*3) {$|V_i| = \kappa_i + \Delta + 2a_i + b_i$};
}

\end{tikzpicture}
\caption{The structure of $G$}
\label{fig:struct}
\end{center}
\end{figure}
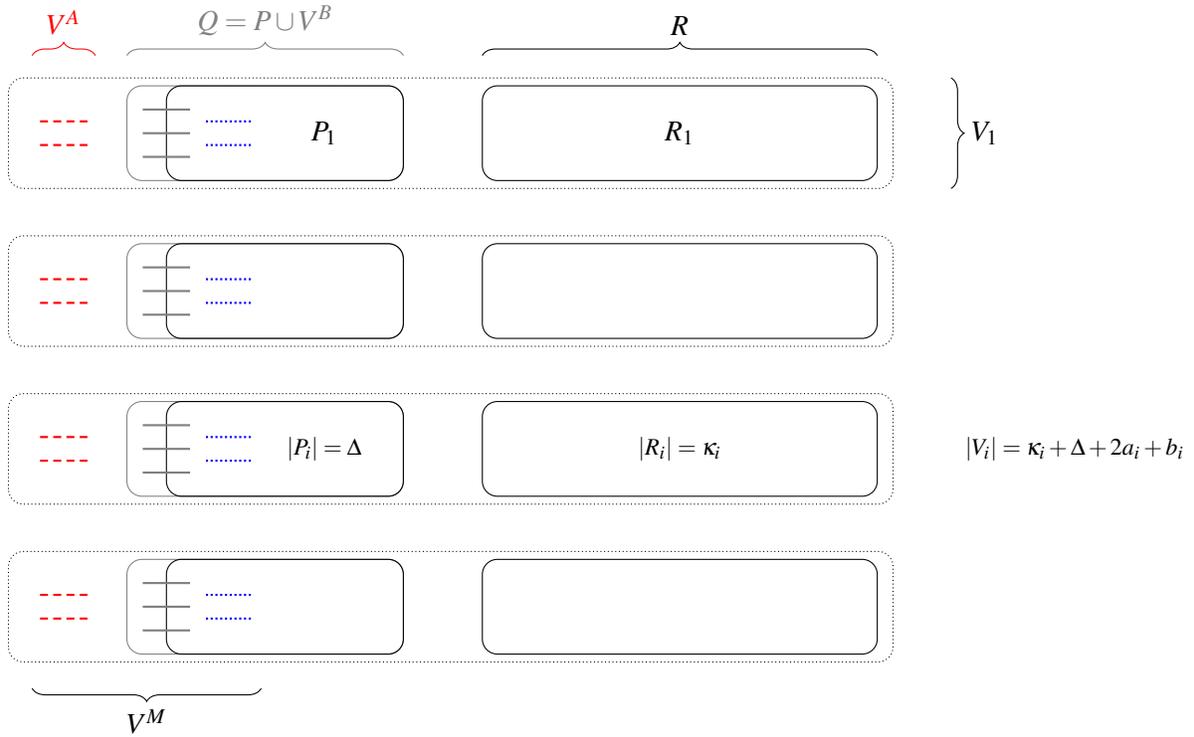

\medskip
We are now ready to introduce our candidate \ptg{r} graph that will satisfy \Cref{thm:weakpentstab}. Let $\tG = L_r[k,\Delta,\Delta,n_1,\dots,n_r]$ be the graph on vertex set $X\cup Y_1 \cup Y_2\cup Z_1\cup \dots \cup Z_r$ as defined in the introduction, where $|X|=k$, $|Y_1|=|Y_2|=\Delta$, $|Z_j|=n_j = \kappa_j+\Delta+2a_j+b_j$ for $j\ge 3$, and
\begin{align*}
	|Z_1| &= n_1 = \kappa_1 + a_1 - c_1\\
	|Z_2| &= n_2 = \kappa_2 + a_1+b_1+c_1+2a_2+b_2- k.
\end{align*}
Note that $|V_j|=|Z_j|$ for $j\ge 3$, and $|V_1|+|V_2| = |X|+|Y_1|+|Y_2|+|Z_1|+|Z_2|$, so $G$ and $\tG$ have an equal number of vertices.

The proof of \Cref{thm:weakpentstab} therefore reduces to establishing \Cref{prop:compare} below.
\end{proof}

\begin{prop} \label{prop:compare}
$\tG$ satisfies both $e(\tG) \ge e(G)$ and $D_r(\tG) \ge D_r(G)$.
\end{prop}

\subsection{Comparing $G$ and $\tG$}

\begin{proof}[Proof of \Cref{prop:compare}]
\Cref{lem:ptgcount} and \eqref{eq:consts} give $D_r(\tG) = k\Delta$. Here the definition of $k$ implies $k\Delta\ge e(H)$ and we clearly have $e(H) \ge D_r(G)$, thus $D_r(\tG)\ge D_r(G)$, and $\tG$ has at least as many internal edges in its standard $r$-partition as $G$. It is therefore enough to show that $\tG$ also has at least as many crossing edges as $G$. We start with a lower bound for $\tG$.

\begin{prop}\label{yaycounting}
The number of crossing edges in $\tG$ is at least
\[
 \sum_{i<j} |V_i||V_j| - \left(\Delta^2 + b_1b_2 + (a_1+b_1+c_1)\kappa_2 + (k-a_1-b_1-c_1)\kappa_1 +  (a_1+a_2)\frac{\Delta}{2}\right).
\]
\end{prop}
\begin{proof}
First of all, as $|V_1\cup V_2| = |Z_1 \cup Z_2\cup Y_1\cup Y_2\cup X|$, and $|V_i|=|Z_i|$ for every $i\ge 3$, there are exactly $\sum_{i<j} |V_i||V_j| - |V_1||V_2|$ crossing edges in $\tG$ incident to $\bigcup_{i=3}^r Z_i$.

As for the edges induced by $Z = Z_1 \cup Z_2\cup Y_1\cup Y_2\cup X$, there are 
\begin{align*}
(|Y_1|+|Z_1|)(|X|+|Y_2|+|Z_2|) &= \big(|V_1| - (a_1+b_1+c_1)\big)\big(|V_2|+ (a_1+b_1+c_1)\big) \\
&= |V_1||V_2| - (|V_2|-|V_1| + a_1+b_1+c_1) (a_1+b_1+c_1)
\end{align*}
potential crossing edges in the standard $r$-partition of $\tG$ (see \Cref{fig:ptg}), out of which
\[
|Y_1||Y_2| + |X||Z_1| = \Delta^2 + k(\kappa_1+a_1-c_1)
\]
are missing. Here $|V_2|-|V_1| + a_1+b_1+c_1 = \kappa_2-\kappa_1 +2a_2 + b_2 - a_1+c_1$, so by rearranging, we get that the number of crossing edges in $\tG$ is
\[ 
\sum_{i<j} |V_i||V_j| - \big( \Delta^2 + b_1b_2 + (a_1+b_1+c_1)\kappa_2 + (k-a_1-b_1-c_1)\kappa_1 + \Lambda \big),
\]
where
\[
\Lambda = a_1(k+2a_2+b_2-a_1-b_1) + a_2(2b_1+2c_1) - c_1 (k-b_1-b_2-c_1) \le (a_1+a_2)\cdot 3k,
\]
where we used that $k=a+b+c\ge a_2+b_1+b_2+c_1$. The result then follows from $\Delta\ge 6k$.
\end{proof}

Recall that there are exactly $\sum_{i<j} |V_i||V_j|$ potential crossing edges in $G$. It therefore suffices to show that at least 
\begin{align}
	\Delta^2+b_1b_2 + (a_1+b_1+c_1)\kappa_2 + (k-a_1-b_1-c_1)\kappa_1 + (a_1+a_2)\frac{\Delta}{2} \label{wanted}
\end{align}
of them are missing from $G$.

It will be easier to split the graph into two, and bound the number of missing edges separately. Let $Q_i = P_i \cup V^B_i$ be the set obtained by extending $P_i$ with the vertices of the matching $B_i$ for every $i\in[r]$ (see \Cref{fig:struct}), so that $V^A_i, Q_i$ and $R_i$ partition $V_i$, and let $Q=\bigcup_{i\in [r]} Q_i$. We first count the number of crossing edges with both endpoints in $Q$, and then the ones with at most one end in $Q$.

\begin{lem}\label{doubleup}
$G$ misses at least $\Delta^2 + b_1b_2$ of the crossing edges induced by $Q$.
\end{lem}
\begin{proof}
We use a similar argument to the proof of \Cref{folklore}. Let $\cF$ be the family of all $r$-sets $\{v_1,\dots,v_r\}$ such that $v_i\in P_i$ for every $i=1,\dots,r$, but $v_1\notin V^B_1$ or $v_2\notin V^B_2$. Then $|\cF| = \Delta^r-b_1b_2\Delta^{r-2}$. Similarly, let $\cG$ be the family of all $(r+2)$ sets $\{v_1,\dots,v_r, v'_1,v'_2\}$ such that $v_1v_1'\in B_1$, $v_2v_2'\in B_2$, and $v_i\in P_i$ for every $i=3,\dots,r$. Then $|\cG| = b_1b_2\Delta^{r-2}$.

Recall that $P_1,\dots, P_r$ were all in the neighborhood of some vertex $u$. This means that there must be a (crossing) edge missing in $G[X]$ for every $X\in \cF$. Also, for $Y\in\cG$, $G[Y]$ is a $K_{r+1}$-free graph on $r+2$ vertices and so must be missing at least two edges. As $v_1v'_1$ and $v_2v'_2$ are both present in $G$, the missing edges in $G[Y]$ are also crossing.

Summing over the sets in $\cF\cup \cG$ gives at least $\Delta^r + b_1b_2\Delta^{r-2}$ missing crossing edges in total. It is easy to check that each missing edge $v_iv_j$ (or $v'_iv_j$ or $v'_iv'_j$) in $G$ is contained in exactly $\Delta^{r-2}$ sets from $\cF\cup\cG$, so $G[Q]$ misses at least $\Delta^2+b_1b_2$ crossing edges.
\end{proof}

\begin{lem}\label{restmissing}
$G$ misses at least 
\begin{equation} \label{eq:remain}
	(a_1+b_1+c_1)\kappa_2 + (k-a_1-b_1-c_1)\kappa_1 + (a_1+a_2)\Delta/2
\end{equation}
crossing edges with at most one endvertex in $Q$.
\end{lem}
\begin{proof}
As a first attempt, we try to find a set of missing crossing edges for each matching edge in $M$ so that they are all disjoint and not induced by $Q$. More specifically, we want to show that for every edge $e\in M_1$, there are $\kappa_2$ missing edges between $e$ and $R=\bigcup_{i\in [r]} R_i$, and for every remaining edge $e\in M\setminus M_1$, there are $\kappa_1$ missing edges between $e$ and $R$. Moreover, for every $e\in A_1\cup A_2$, we want $\Delta/2$ additional missing edges between $e$ and $Q$. As $|M|=k$ and $|M_1|=a_1+b_1+c_1$, this would be exactly the amount we need.\footnote{The reader might find it helpful to check what the bound means when $G$ is a \ptg{r} graph: the $r$-partition $V_1\cup\dots\cup V_r$ is much like the standard $r$-partition, except the set $X$ might be split between $V_1$ and $V_2$. In any case, we always have $M = B_1\cup B_2$ (in particular, $a_1=a_2=0$), and every edge in $B_i$ contributes exactly $\kappa_{3-i}$ missing edges: one to each vertex of $R_{3-i}$.}

Of course, it may well be that some edge in $M$ is incident to fewer missing edges. Let $M'_1 = M_1$ and $M'_2 = M\setminus M_1$. To first bound the number of crossing edges between $M$ and $R$, we define $\tau$ to be the largest ``deficit'' in the above counting, i.e., the smallest \emph{nonnegative} integer such that for each $i=1,2$ and every edge $vv'\in M'_i$, there are at least $\kappa_{3-i}-\tau$ missing edges between $\{v,v'\}$ and $R\setminus R_i$.

To count the missing edges between $A_1\cup A_2$ and $Q$, we split $A_i$ into $A^g_i \cup A^b_i$ for each $i=1,2$ as follow. $A^g_i$ is the set of ``good'' edges $vv'$, such that there are at least $\kappa_{3-i}-\tau+\Delta/2$ edges missing between $\{v,v'\}$ and $(Q\cup R) \setminus (Q_i\cup R_i)$, and $A^b_i$ is the set of ``bad'' edges, where this is not the case. 

So far this gives at least
\[ 
 |A^g_1|(\kappa_2-\tau) + |A^g_2|(\kappa_1-\tau) + (|A^g_1| + |A^g_2|)\Delta/2
\] 
missing crossing edges between the good edges of $A$ and $Q\cup R$, and another
\[
 (|M_1'|-|A^g_1|)(\kappa_2-\tau) + (|M_2'|-A^g_2)(\kappa_1-\tau)
\]
between all other edges of $M$ and $R$. This is a total of
\begin{equation} \label{eq:stdedges}
 |M_1'|(\kappa_2-\tau) + |M_2'|(\kappa_1-\tau) + (|A^g_1| + |A^g_2|)\Delta/2
\end{equation}
missing edges between $V^M$ and $R$. To get \eqref{eq:remain}, we need to analyze the structure a bit. 

\medskip
Let $vv'\in A^b_i$ be some fixed bad edge for some $i$. Then there are at most $\kappa_{3-i}-\tau+\Delta/2$ missing edges from $\{v,v'\}$ to $(R\cup Q)\setminus (R_i\cup Q_i)$, and by the definition of $\tau$, at least $\kappa_{3-i}-\tau$ of these are incident with $R\setminus R_i$. So $vv'$ must have at least $\Delta/2$ common neighbors in each $P_j$ with $j\ne i$. In particular, as $k \le \Delta/6$ and hence $|V^M|=2k\le \Delta/3$, we get that for every $j\ne i$ there is a set $N_j\subs P_j\setminus (V^B_j\cup V^C_j)$ of at least $\Delta/6$ common neighbors in $P_j$ that is disjoint from $V^M$.

Choose $i'\ne i$ so that $\Gamma(v)\cap \Gamma(v')\cap R_{i'}$ is smallest. Then for every $j\ne i,i'$, 
\begin{align} \label{eq:commonneighbors}
 |\Gamma(v)\cap \Gamma(v')\cap R_j| &\ge \frac{ |\Gamma(v)\cap \Gamma(v')\cap (R_{i'}\cup R_j)|}{2} \nonumber \\
 &\ge \frac{(\kappa_{i'} + \kappa_j) - (\kappa_{3-i}+\Delta/2)}{2} \ge \frac{\kappa_2}{2} - \frac{\Delta}{4} \ge \frac{7\kappa_2}{16}
\end{align}
because there are at most $\kappa_{3-i}+\Delta/2$ missing edges from $\{v,v'\}$ to $R_{i'}\cup R_j$, and we also used $\kappa_{i'}+\kappa_j\ge \kappa_2+\kappa_{3-i}$ and $\Delta\le \kappa_2/4$, which follow from \eqref{eq:consts} for any distinct $i\in\{1,2\}$, $i'$ and $j$. 

\begin{obs}
	We may assume that every triangle induced by $V_i\cup V_i'$ has at most $\kappa_2/4$ common neighbors in some $R_j$ with $j\ne i,i'$.
\end{obs}
Indeed, the common neighborhood of this triangle is $\kappa_{r-2}$-free. The case $r\le 3$ is then vacuously true, so suppose $r \ge 4$. Then if the triangle has at least $\kappa_2/4$ common neighbors in every $R_j$ with $j\ne i,i'$, then by \Cref{folklore}, $G[R]$ misses at least $\kappa_2^2/16$ crossing edges. But $\kappa_2^2/16 \ge k(\kappa_2+\Delta) \ge \eqref{eq:remain}$, so we are done.

\medskip
This means that for the above bad edge $vv'$, we can assume that every triangle $vv'w$ with $w\in N_{i'}$ has at most $\kappa_2/4$ common neighbors in some $R_j$ with $j\ne i,i'$. Using \eqref{eq:commonneighbors}, we see that there are at least $\frac{7\kappa_2}{16}-\frac{\kappa_2}{4} = \frac{3\kappa_2}{16} > 4k$ missing edges between $w$ and $R\setminus (R_i\cup R_j)$. Summing over all $w\in N_{i'}$, we find at least
\begin{equation} \label{eq:badedges}
 4k\Delta/6 \ge k\Delta/2 \ge (|A^b_1| + |A^b_2|)\Delta/2
\end{equation}
missing edges between $Q\setminus V^M$ and $R$.

\medskip

If $\tau=0$, then we are already done: \eqref{eq:stdedges} and \eqref{eq:badedges} together give enough edges for \eqref{eq:remain}. So let us assume that $\tau>0$, i.e., there is an edge $vv'\in M'_i$ for some $i$ such that there are exactly $\kappa_{3-i}-\tau$ missing edges between $\{v,v'\}$ and $R\setminus R_i$.

Once again, choose $i'\ne i$ so that $\Gamma(v)\cap \Gamma(v')\cap R_{i'}$ is smallest. Then, similarly to \eqref{eq:commonneighbors},
\[
 |\Gamma(v)\cap \Gamma(v')\cap R_{i'}| \ge \kappa_{i'} - (\kappa_{3-i} - \tau) \ge \tau
\]
 and for every $j\ne i,i'$,
\[ 
 |\Gamma(v)\cap \Gamma(v')\cap R_j| \ge \frac{\kappa_{i'} + \kappa_j - (\kappa_{3-i} - \tau)}{2} \ge \frac{\kappa_2}{2}.
\]
By \Cref{folklore}, there must be at least
\begin{equation} \label{eq:Redges}
\frac{\kappa_2}{2} \cdot \tau \ge k\tau
\end{equation}
missing edges induced by $R$. Adding \eqref{eq:stdedges}, \eqref{eq:badedges} and \eqref{eq:Redges} together, we get \eqref{eq:remain}.
\end{proof}

Putting \Cref{yaycounting} and \Cref{doubleup,restmissing} together yields \Cref{prop:compare}, and finishes the proof of our main result.
\end{proof}

\section{Concluding remarks} \label{sec:conclusion}

With \Cref{pentstab} in hand, finding the exact \ptt graph $G$ that maximizes $D_r(G)$ assuming $e(G)\ge t_r(n) - \delta n^2$ is a matter of calculation. The result of Balogh, Clemen, Lavrov, Lidick\'y and Pfender \cite{BCLLP} shows that among \ptt graphs with $t_r(n)-\delta n^2$ edges, $D_r(G)$ is maximized when $x\approx \frac{2r}{3}\delta n$, $y\approx \sqrt{\frac{\delta}{3}}n$, $n_j\approx (\frac{1}{r}+\frac{2}{3}\delta)n$ for $j\ge 3$, and $n_i\approx (\frac{1}{r}-\frac{2(r-1)}{3}\delta - \sqrt{\frac{\delta}{3}})n$ for $i=1,2$, and the maximum is $D_r(G)\approx \frac{2r}{3\sqrt{3}}\delta^{3/2}n^2$.

\medskip

It would be very interesting to find exact stability results for other classes of graphs. Of course, this is generally a harder problem than determining the exact extremal graphs, which is often already a difficult task on its own.
A natural next step is to consider $H$-free graphs where $H$ is a graph with a critical edge, that is, there is an edge $e \in E(H)$ such that the deletion of $e$ from $H$ reduces the chromatic number. Examples of such graphs include cliques and odd cycles.

An old theorem of Simonovits \cite{S74} says that when $H$ is an $(r+1)$-chromatic graph with a critical edge, the Tur\'an graph $T_r(n)$ is the unique $H$-free graph maximizing the number of edges, provided $n$ is large enough. But even in this case, it seems unclear what the right conjecture should be for the set of $H$-free graphs $G$ that maximize $D_r(G)$ when $e(G)\ge t_r(n)-t$. We think that the theorem of Erd\H{o}s, Gy\H{o}ri and Simonovits should at least generalize to odd cycles in the following sense: Among $C_{2k-1}$-free graphs of close to extremal size, some $C_{2k+1}$-blowup is farthest from being bipartite.

Unfortunately, this might fail when the number of edges is very close to the extremal number. For example, let $G$ be the graph obtained from $C_6[1,1,1,1,n/2-2,n/2-3]$ by adding a vertex adjacent to the first three (singly blown up) vertices. Then $G$ is a $C_5$-free graph satisfying $D_2(G)=1$, but with strictly more edges than any blowup of $C_7$ (itself being a supergraph of the densest $C_7$-blowup). Nevertheless, we believe that the existence of such examples is an artifact of the small blowup factors, and $C_{2k+1}$-blowups are still optimal when the density of $G$ is bounded away from $1/4$.
\begin{conj}
Fix $k\ge 2$ and let $\delta$ be small enough. Then for any $\delta>\delta_0>0$ and large enough $n$, the following holds. For every $C_{2k-1}$-free graph $G$ on $n$ vertices with $(\frac{1}{4}-\delta_0)n^2\ge e(G)\ge (\frac{1}{4}-\delta)n^2$ edges, there is a $C_{2k+1}$-blowup $\tG$ satisfying $e(\tG)\ge e(G)$ and $D_2(\tG) \ge D_2(G)$.
\end{conj}
Blowups of $C_{2k+1}$ might also be optimal for every 3-chromatic graph $H$ with a critical edge, whose shortest odd cycle has length $2k-1$. Such graphs are certainly $H$-free, and results of Roberts and Scott \cite{RS18} imply that the bound they give on $D_2(G)$ (with $e(G)$ fixed) is tight up to a constant factor. 

It is also tempting to guess that when $H$ is a general $(r+1)$-chromatic graph with a critical edge, then the optimum $D_r(G)$ is attained by complete $C_{2k+1}$-Tur\'an graphs (defined analogously to \ptt graphs by inserting a blowup of $C_{2k+1}$ into a part of a complete $(r-1)$-partite graph), where $k$ is some parameter depending only on $H$.

\medskip

A closely related problem, which served as the main motivation for the paper of Erd\H{o}s, Gy\H{o}ri and Simonovits \cite{erdoscan'tmaths}, is the old conjecture of Erd\H{o}s \cite{E76} claiming $D_2(G) \le \frac{n^2}{25}$ for every $K_3$-free graph $G$ on $n$ vertices. This trivially holds when $e(G)\le \frac{2n^2}{25}$, and was proved for $e(G)\ge \frac{n^2}{5}$ by Erd\H{o}s, Faudree, Pach and Spencer \cite{EFPS88}. If true, the conjecture is tight for a balanced blowup of $C_5$.

This problem led to further research into how far $K_{r+1}$-free graphs can be from being bipartite. Sudakov \cite{S07} proved a variant of the conjecture for 4-cliques, showing that $D_2(G)$ is maximized by $G=T_3(n)$ among $K_4$-free graphs. Sudakov conjectured that this generalizes to larger cliques (i.e., among $K_{r+1}$-free graphs, $D_2(G)$ is maximum when $G=T_r(n)$). A proof of this for $K_6$ has been announced by Hu, Lidick\'y, Martins, Norin and Volec \cite{HLMNV}. The remaining cases remain wide open.

%%% AUTHOR:
%%% Bibliography goes here. Note that the arXiv cannot process bibtex
%%% or biber bibliographies.  Example of acceptable bibliograpy format:
\bibliographystyle{amsplain}

%% AUTHOR: You can generate such a bibliography from a .bib file by 
%% running pdflatex/bibtex/pdflatex/pdflatex and then pasting the .bbl file
%% between \begin{thebibliography} and \end{bibliography}

%%% AUTHOR: Include a short description of each author following the
%%% structure below. Use the same short tags used previously.  
%%% Use \imageat{} and \imagedot{} instead of "@" and "." in
%%% email addresses-this replaces the symbols with graphics to avoid 
%%% e-mail address harvesting from the .pdf file
\begin{aicauthors}
\begin{authorinfo}[dk]
  D\'aniel Kor\'andi\\
  Mathematical Institute, University of Oxford\\
  Oxford, United Kingdom\\
  korandi\imageat{}maths\imagedot{}ox\imagedot{}ac\imagedot{}uk \\
  \url{https://korandi.org/}
\end{authorinfo}
\begin{authorinfo}[ar]
  Alexander Roberts\\
  Mathematical Institute, University of Oxford\\
  Oxford, United Kingdom\\
  robertsa\imageat{}maths\imagedot{}ox\imagedot{}ac\imagedot{}uk \\
  \url{https://people.maths.ox.ac.uk/robertsa/}
\end{authorinfo}
\begin{authorinfo}[as]
  Alex Scott\\
  Mathematical Institute, University of Oxford\\
  Oxford, United Kingdom\\
  scott\imageat{}maths\imagedot{}ox\imagedot{}ac\imagedot{}uk \\
  \url{https://people.maths.ox.ac.uk/scott/}
\end{authorinfo}
\end{aicauthors}

\end{document}